\documentclass{article}
\advance\textheight -\headheight
\advance\textheight -\headsep
\oddsidemargin\paperwidth
\advance\oddsidemargin -\textwidth
\divide\oddsidemargin2 
\ifdim\oddsidemargin<.5truein \oddsidemargin.5truein \fi
\advance\oddsidemargin -1truein
\evensidemargin\oddsidemargin
\topmargin\paperheight \advance\topmargin -\textheight
\advance\topmargin -\headheight \advance\topmargin -\headsep
\divide\topmargin2
\ifdim\topmargin<.5truein \topmargin.5truein \fi
\advance\topmargin -1.4truein\relax

\usepackage{amscd,amsthm, amsmath, amsfonts,amssymb,epsfig} 
\usepackage{amsfonts}
\usepackage{color,graphics}
\usepackage[usenames,dvipsnames]{pstricks}
\usepackage{pst-grad} 
\usepackage{pst-plot} 

\makeatletter

\@addtoreset{equation}{section}

\newtheorem{theorem}{Theorem}[section]
\newtheorem{lemma}[theorem]{Lemma}

\newtheorem{corollary}[theorem]{Corollary}
\theoremstyle{definition}
\newtheorem{de}[theorem]{Definition}

\newtheorem{ex}[theorem]{Example}
\newtheorem{re}[theorem]{Remark}
\newtheorem{konst}[theorem]{Construction}
\newtheorem{definition}[theorem]{Definition}

\makeatother

\def\e{\begin{equation}}
	\def\ee{\end{equation}}

\def\1{\frac{1}{2}}

\def\bP{\mathbb{CP}^1}
\def\oMn{\overline{{\mathcal M}_{0,n}^{\mathbb R}}}

\def\R{{\mathbb R}}

\def\z{\zeta}

\begin{document}

	\title{Belyi pair for the orientation cover  of~$  \overline{{\mathcal M}_{0,5}^{\mathbb R}}$ }
	\author{N. Ya.  Amburg, E. M. Kreines}
	\date{}
	
	\maketitle
	
	\hfill To G.B. Shabat with admiration on the occasion of his 70th birthday.
	
	\begin{abstract}
		Let $\overline{{\mathcal M}_{0,5}^{\mathbb R}}$ be the Deligne-Mumford compactification of the moduli space of genus 0 real algebraic curves with 5 marked points. By ${\mathcal L}(\overline{{\mathcal M}_{0,5}^{\mathbb R}})$ we denote its orientation cover. 
		The  cell decomposition of  ${\mathcal L}(\overline{{\mathcal M}_{0,5}^{\mathbb R}})$ is a dessin d'enfant of the genus 4.  In this paper we compute the Belyi pair for this dessin. In particular, it appears that the corresponding curve is the celebrated Bring curve.
		
		{\bf Key words}: dessins d'enfants, Belyi functions, Bring curve, moduli space of  real algebraic curves
	\end{abstract}
	
	\tableofcontents

	\section{Introduction}

	Let $\overline{{\mathcal M}_{0,n}^{{\mathbb R}}}$  be the Deligne-Mumford compactification of the moduli space
	of real algebraic curves of genus 0 with $n$ marked and numbered  points. The variety $\overline{{\mathcal M}_{0,n}^{{\mathbb R}}}$ was  investigated, see, for example, \cite{Dev,EtingofCo,Kap}. In particular,  in \cite{AmbKre,Ceyhan} the first  Stiefel-Whitney  class for   $\overline{{\mathcal M}_{0,n}^{\mathbb R}}$ was computed and a geometric description of the dual class was obtained.

	There exists a natural cellular decomposition of $\overline{{\mathcal M}_{0,n}^{\mathbb R}}$. 
	Orientation cover of $\overline{{\mathcal M}_{0,5}^{\mathbb R}}$ is a compact smooth oriented surface of the genus 4, so the cellular decomposition provides a dessin d'enfant, i.e., a  connected embedded graph of a certain special structure, see Definition~\ref{DefDD}.  
	The theory of dessins d'enfants was initiated by A.~Grothendieck
	in~\cite{G} and actively developed thereafter,
	see~\cite{LZ} and references therein. 
	Dessins d'enfants are naturally related to so-called Belyi pairs, i.e., non-constant meromorphic functions with at most 3 critical values defined on algebraic curves, see Definition~\ref{DefBP}. These relations provide plenty of non-trivial applications in algebra, geometry, mathematical physics, etc. It is a difficult important problem to compute precisely the Belyi pair corresponding to a given dessin. There is just a small amount of computed examples, especially in positive genuses.
	In this paper we compute the Belyi function of the genus 4 dessin d'enfant given by the natural  cell decomposition of the orientation cover of $\overline{{\mathcal M}_{0,5}^{\mathbb R}}$. 
	
	Bring curve is known since 1786, see \cite{Bring}, and was widely explored further, see \cite{Zv} and references therein. While reading the paper \cite{Zv} we observed that the permutation group $S_5$ acts on a certain dessin d'enfant of genus 4, 
	we wandered  if there are two different dessins of genus 4 with such large symmetry group. We are glad that we managed to connect Bring curve with the dessin obtained from the natural cellular decomposition of  ${\mathcal L}(\overline{{\mathcal M}_{0,5}^{\mathbb R}})$, and prove that this dessin is dual to $I_4\cup I_4^*$ introduced in \cite{Zv}, see also Definition \ref{defI_4} in this text. 
	
	Our work is organized as follows: Section 2 contains necessary definitions and general description of the cell decomposition of the variety   $\overline{{\mathcal M}_{0,n}^{\mathbb R}}$. In Section 3 we provide  the details concerning the cell decomposition ${\mathcal D}$ of  the orientation cover  ${\mathcal L}(\overline{{\mathcal M}_{0,5}^{\mathbb R}})$ of $\overline{{\mathcal M}_{0,5}^{\mathbb R}}$, which includes the number of cells, adjacency types, and graphic illustration of the adjacency. It is shown that ${\mathcal D}$ is a dessin d'enfant, its genus and combinatorial type are evaluated. In  Section 4 we introduce the genus 4 icosahedron, discuss its symmetries and Belyi functions on the Bring curve. Section 5 contains the final formula of the Belyi function corresponding to the dessin d'enfant arised from this cell decomposition of the orientation cover~${\mathcal L}(\overline{{\mathcal M}_{0,5}^{\mathbb R}})$.
	
	\section{Deligne-Mumford compactification of    ${\mathcal M}_{0,n}^{\mathbb R}$} 
	
	In this and the next sections we basically follow definitions and notations from our paper \cite{AmbKre}. We provide here the detailed description of a well-known construction of real curve moduli space cell decomposition   for the sake of completeness and convenience of a reader.
	
	Let  $n\ge 3$. The moduli space of genus 0 real algebraic curves with $n$ marked and numbered points, i.e., the real algebraic variety parametrizing the isomorphism classes of these curves, is denoted by ${\mathcal M}_{0,n}^{\mathbb R}$. 
	
	\subsection{Stable curves}
	
	Following the paper \cite{EtingofCo}, we represent real stable curves as ``cacti-like'' structures, consisting of   circles with the points $\{1,2,\ldots,n\}$ on them:
	
	\begin{definition} \label{SC} \cite{DM}
		A {\em stable curve\/} of genus 0 with $n$ marked and numbered points over the field of real numbers  $\R$ is a finite union of real projective lines 
		$C=C_1\cup C_2\cup \ldots\cup C_p$  with $n$ different marked points  $z_1,z_2,\ldots,z_n\in C$, if the following conditions hold.
		\begin{enumerate}
			\item For each point  $z_i$ there exists the unique line  $C_j$, such that $z_i\in C_j$.
			\item For any pair of lines $C_i\cap C_j$ is either empty or consists of one point, and in the latter
			case the intersection is transversal.
			\item \label{it3} The graph corresponding to  $C$ (the lines $C_1,C_2,\ldots,C_p$ correspond to the vertices; two vertices are incident to the same edge iff corresponding lines have non-empty intersection) is a tree.   
			\item  \label{it4} The total number of special points (i.e. marked points or intersection
			points) that belong to a given $C_j$ is at least 3 for each $j=1,\ldots,p$.
		\end{enumerate}
		We say that  $p$ is the  {\em number of components\/} of the stable curve, and $C_j$ is a {\em component\/}.
	\end{definition}
	\begin{figure}[h]
		\centering
		\includegraphics[width=.35\linewidth]{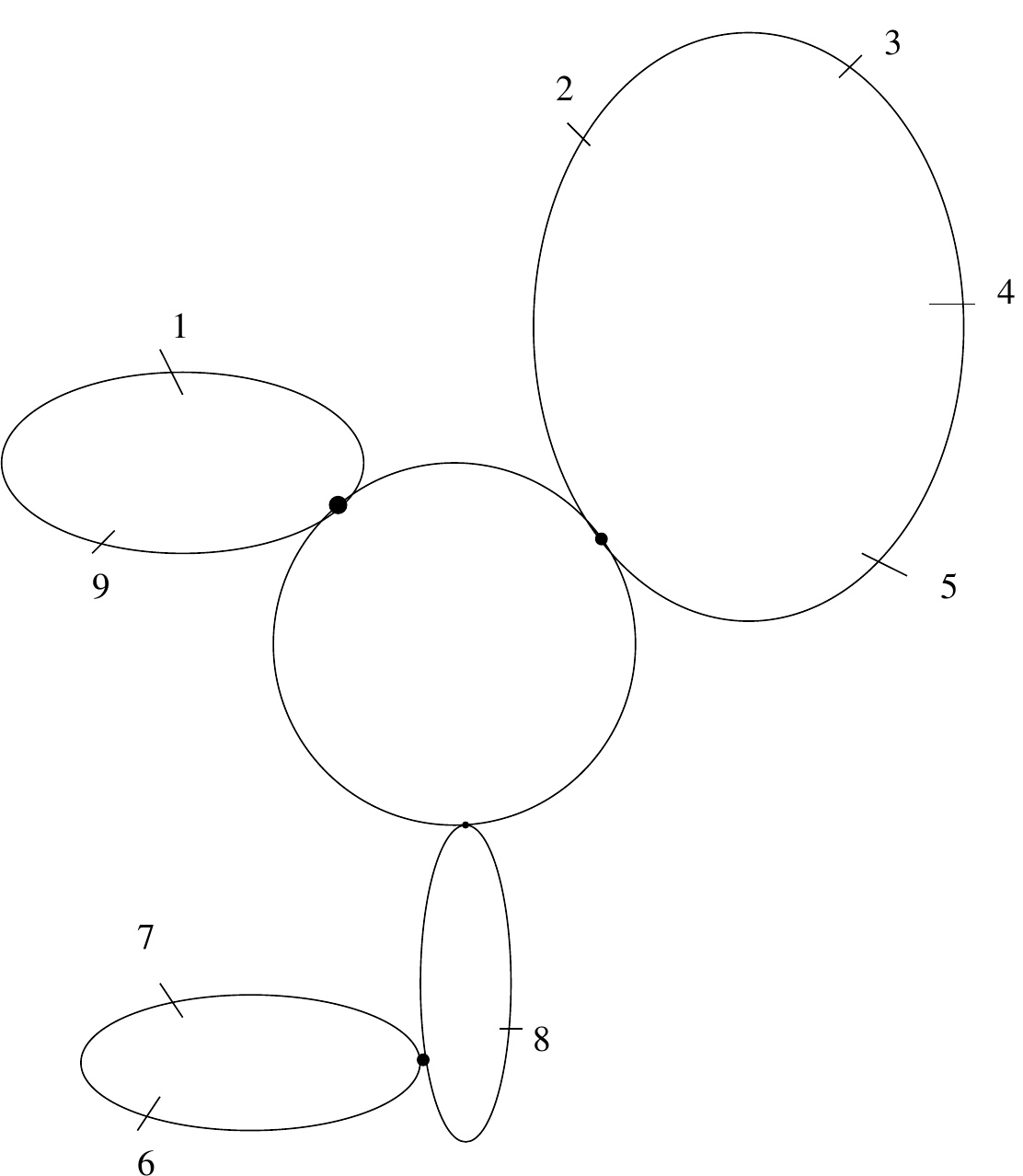}
		\caption{A stable curve over  $\mathbb R$ of genus  $0$ with $9$ marked points}
		\label{exstcurve1}
	\end{figure}
	\begin{definition}
		Let  $C=(C_1,C_2,\ldots,C_p,z_1,z_2,\ldots,z_n)$ and $C'=(C'_1,C'_2,\ldots,C'_p,z'_1,z'_2,\ldots,z'_n)$ be real stable curves of genus 0 with $n$ marked and numbered points. $C$ and $C'$ are called {\em equivalent\/} if there exists an isomorphism of algebraic curves $f:C\to C'$ such that   $f(z_i)=z'_i$ for all $i=1,\ldots , n$.
	\end{definition}
	
	\subsection{Moduli space $\overline{{\mathcal M}_{0,n}^{\mathbb R}}$}
	
	\begin{definition}
		Let  $n\ge 3$. {\em Deligne-Mumford compactification\/} $\overline{{\mathcal M}_{0,n}^{\mathbb R}}$   of the moduli space of genus 0 real algebraic curves with $n$ marked and numbered points 
		is the set of equivalence classes of the genus 0 stable curves with $n$ marked and numbered points defined over ${\mathbb R}$.
	\end{definition}
	\begin{theorem}\cite{Dev} For all $n\ge 3$ the set $\oMn$ is a real compact variety of   dimension  $\dim(\oMn)=n-3$ and is the closure of ${\mathcal M}_{0,n}^{\mathbb R}$.
		If $n>4$, then  $\oMn$ is non-orientable.
	\end{theorem}

	\subsection{Cell decomposition of  $\overline{{\mathcal M}_{0,n}^{\mathbb R}}$}
	\begin{re}
		There exists a natural structure of  cell decomposition for the space $\overline{{\mathcal M}_{0,n}^{\mathbb R}}$. This structure is described for example in the works \cite{Dev,Kap}. 
	\end{re}
	
	Following \cite{Dev}, we denote by ${\mathcal G}^L(n,k)$ the set of $n$-gons with the labels  $ 1,2,\ldots,n $ on the edges and with $k$ non-intersecting diagonals.
	
	\begin{de}
		A {\em twist\/} of $M \in {\mathcal G}^L(n,k)$ along a diagonal $d$ is the $n$-gon $M' \in {\mathcal G}^L(n,k)$, obtained  by the following procedure. At first we  cut $M$ along  $d$.  Then one (any one) of the parts is   rotated  around the axis, which is orthogonal to $d$ in the plane of the $n$-gon, by  $180^\circ$. Finally, we glue two parts along~$d$. 
	\end{de}
	\begin{re}
		Let $M'$ be the twist of $M$, let the labels $ 1,2,\ldots,n $ of the sides of  $M$ be ordered as  $\z_1,\ldots,\z_n$, and let the sides marked by $\z_1,\ldots,\z_k$ be separated by $d$ from the sides marked by $\z_{k+1},\ldots, \z_n$. Then the sides of $M'$ have ordered labels $\z_1,\ldots,\z_k,\z_{n}, \z_{n-1},\ldots, \z_{k+1}$.  
	\end{re}

	\begin{konst} \cite{Dev,Kap}  \label{l:oMn_cell_dec}
		{\bf Description of the cell decomposition of~$\overline{{\mathcal M}_{0,n}^{\mathbb R}}$. }
		We label different cells of the moduli space $\overline{{\mathcal M}_{0,n}^{\mathbb R}}$  by  right $n$-gons with marked sides and, possibly,  several non-intersecting diagonals. The  sides  of these $n$-gons correspond to the marked points on a curve and  are also marked by  $ 1,\ldots, n$.   The cells of the maximal dimension are labeled by $n$-gons without diagonals. The cells of codimension 1 are labeled by $n$-gons with one diagonal. Note that these cells consist exactly of 2-component stable curves. The cells of codimension 2, i.e., that correspond to 3-component curves, are labeled by  $n$-gons with 2 diagonals. In general, a cell of codimension $k$ is labeled by an  $n$-gon $M$ with $k$ diagonals. These diagonals divide $M$ into $k+1$ polygons $M_1,\ldots, M_{k+1}$. The edges of $M_1,\ldots, M_{k+1}$, which are the edges of $M$, are labeled by the points marking the  stable curve. Note that the condition  \ref{it3} of Definition \ref{SC} guarantees that different diagonals do not intersect outside the vertices of  $M$. The condition \ref{it4} of Definition \ref{SC} guarantees that each of $M_1,\ldots, M_{k+1}$ has at least 3 sides, i.e., it is a polygon.
		Polygons $M$ and $M'\in {\mathcal G}^L(n,k)$ mark the same cell of the moduli space if $M$ can be transformed to $M'$ by series of twists and the dihedral group of $n$-gon actions.  
	\end{konst}
	\begin{re}
		Special charm of this construction is that marked points and singular points (the points of intersection of different components of a curve) do not have principal differences. Namely, both of them are denoted by edges of a polygon. Also, each component,    or  a connected union of several components, of the curve is denoted by a polygon. This polygon marks the cell of the cell decomposition for the moduli space of a possibly smaller dimension.
	\end{re}
	\begin{ex} \label{ex1}
		Cell decomposition of $\overline{{\mathcal M}_{0,n}^{\mathbb R}}$ contains $\frac{(n-1)!}{2}$ cells of the maximal dimension $n-3$.
	\end{ex}
	\begin{ex}
		$\overline{{\mathcal M}_{0,4}^{\mathbb R}}$ is a circle consisting of 3 cells of the dimension  $1$ and 3 cells of the dimension~$0$.
		Figure \ref{M_0,4} represents the cell decomposition of $\overline{{\mathcal M}_{0,4}^{\mathbb R}}$. Nearby each cell we provide its ``typical'' representative, i.e., one of the stable curves which constitute this cell.
		\begin{figure}[h]
			\centering
			\epsfig{figure=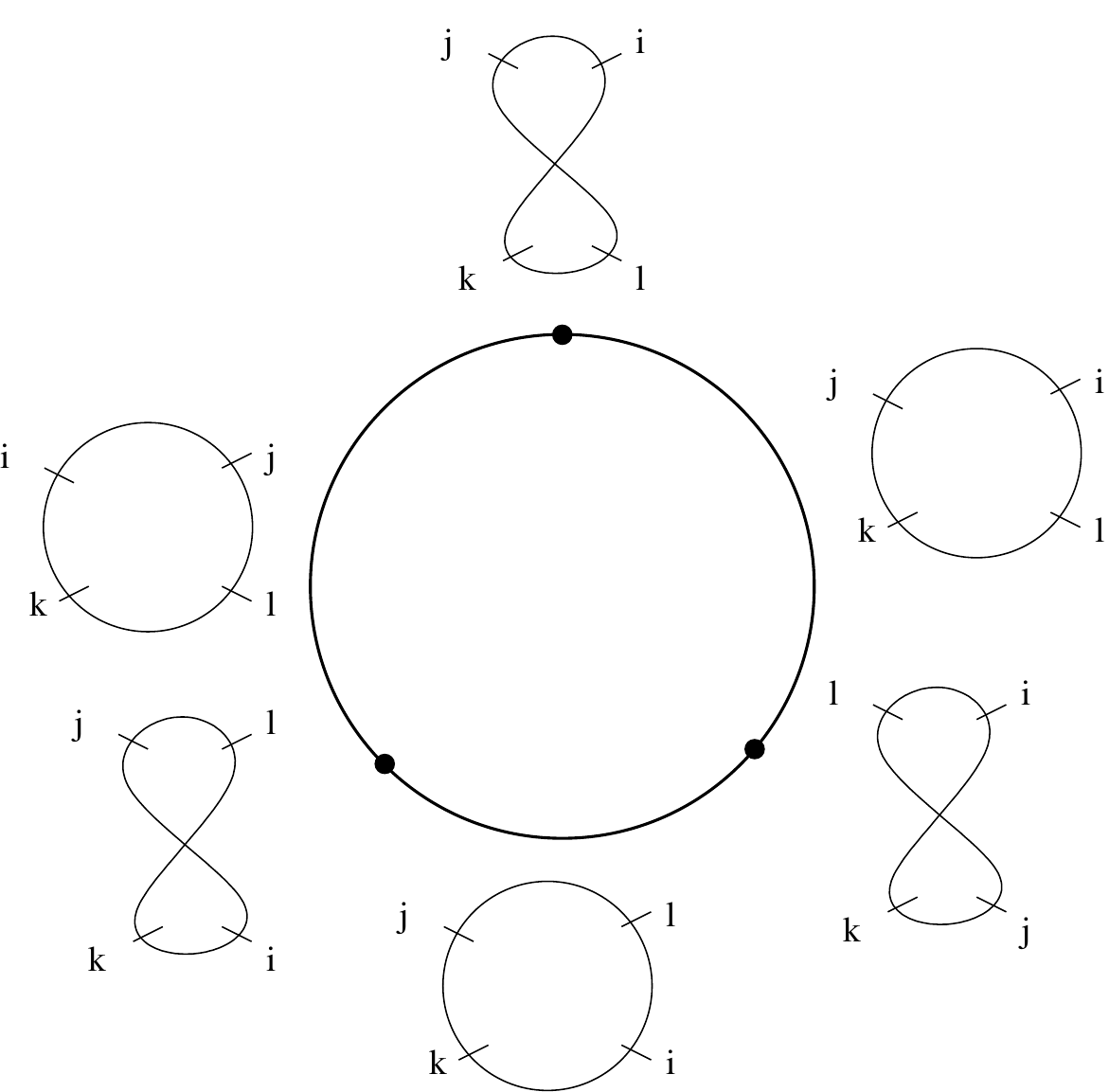,width=.35\linewidth}
			\caption{$\overline{{\mathcal M}_{0,4}^{\mathbb R}}$.}
			\label{M_0,4}
		\end{figure}
	\end{ex}

	\subsection{Orientation cover of  $\overline{{\mathcal M}_{0,n}^{\mathbb R}}$}

	\begin{de} \label{dOC}
		Let ${M}$ be a connected non-orientable manifold. {\em Orientation covering} of ${M}$ is a 2-fold covering $\rho: {\mathcal L}(M) \to {M}$ such that ${\mathcal L(M)}$ is connected and orientable. If $\rho: {\mathcal L}(M) \to {M}$ is an orientation covering, we say that ${\mathcal L(M)}$ is the {\em orientation cover} of~${M}$.
	\end{de}
	
	Correctness of this definition and existence of an orientation cover for any non-orientable manifold are proved for example in~\cite[Proposition 2.2]{Kreck}, unicity is given by \cite[Proposition 2.1]{Kreck}. If $\rho: {\mathcal L(M)} \to {M}$   is the orientation covering, then there exists an orientation converting involution $\pi: {\mathcal L(M)} \to {\mathcal L(M)}$, which permutes the folds, and $\rho \circ \pi =  \rho.$ 
	The detailed and complete description   of the orientation cover for $\overline{{\mathcal M}_{0,n}^{\mathbb R}}$ for all $n\ge 5$ is provided in~\cite[Theorem~6]{DevMor} and in~\cite{AmbKre}. Here we just point out that it inherits the cell decomposition described in the previous section, and each cell of the maximal dimension of $\overline{{\mathcal M}_{0,n}^{\mathbb R}}$ is covered by  two identical cells of ${\mathcal L}(\overline{{\mathcal M}_{0,n}^{\mathbb R}})$ with the opposite orientation.  

	
	It is well-known that the group $S_n$ acts faithfully and smoothly on $\overline{{\mathcal M}_{0,n}^{\mathbb R}}$ by the permutation of the marked points. This action can be naturally lifted to ${\mathcal L}(\overline{{\mathcal M}_{0,n}^{\mathbb R}})$:
	
	\begin{lemma}
		\label{SnD}
		\cite[Theorem~6]{DevMor}
		The symmetric group $S_n$  acts faithfully and smoothly on the orientable cover ${\mathcal L}(\overline{{\mathcal M}_{0,n}^{\mathbb R}})$ by permuting labels.
	\end{lemma}
	
	Under this action some elements of $S_n$ preserve the orientation, and some elements convert the orientation. Below we introduce another action of $S_n$ on ${\mathcal L}(\overline{{\mathcal M}_{0,n}^{\mathbb R}})$ which preserves the orientation.

	\begin{lemma} \label{Sn}
		There exists an action   of the symmetric group $S_n$   on ${\mathcal L}(\overline{{\mathcal M}_{0,n}^{\mathbb R}})$ which  is faithful and for each $\sigma\in S_n$ the  induced map $  {\mathcal L}(\overline{{\mathcal M}_{0,n}^{\mathbb R}}) \to {\mathcal L}(\overline{{\mathcal M}_{0,n}^{\mathbb R}})$ is an orientation preserving diffeomorphism.  
	\end{lemma}
	
	\begin{proof}
		Let $\sigma\in S_n$ and $g_\sigma$ denotes the  automorphism of ${\mathcal L}(\overline{{\mathcal M}_{0,n}^{\mathbb R}})$ determined by $\sigma$ according to Lemma~\ref{SnD}. 
		Then $\pi \circ g_\sigma = g_\sigma \circ \pi$ since $g_\sigma$ is lifted from the action on $\overline{{\mathcal M}_{0,n}^{\mathbb R}}$.
		It follows that $\pi$ is not induced by the $S_n$ action on ${\mathcal L}(\overline{{\mathcal M}_{0,n}^{\mathbb R}})$.
		Therefore the group $G=S_n\oplus {\mathbb Z}_2$ acts on ${\mathcal L}(\overline{{\mathcal M}_{0,n}^{\mathbb R}})$. Namely, the pair $(\sigma,a)\in G$ acts as $g_\sigma\circ \pi^a$, where   $a\in \{0,1\}$. This action is faithful and each element of the group acts as a diffeomorphism since it is a composition of two diffeomorphisms. 
		
		The map $\pi$ changes the orientation and $\pi^2=e$ preserves the orientation. Therefore, the map $g_\sigma \to g_\sigma\circ\pi$   is the bijection between the subsets of orientation preserving and orientation converting elements of $G$. Hence, exactly a half of the elements of $G$ preserves the orientation. 
		
		Let us consider the subset $G' \subset S_n\oplus {\mathbb Z}_2$:
		$$G'=\left\{ \begin{array}{cc} (\sigma,0), & \mbox{ if } g_\sigma \mbox{ preserves the orientation}, \\ (\sigma,1), & \mbox{ if } g_\sigma \mbox{ does not preserve the orientation}. \end{array}\right.  $$
		Observe that for $(\sigma_1,a_1), (\sigma_2,a_2)\in G'$ we have $(\sigma_1,a_1) (\sigma_2,a_2)=(\sigma_1 \sigma_2,a_1+a_2)\in G'$. Indeed, if $\sigma_1$ and $\sigma_2$ are both orientation preserving or both orientation converting, then $\sigma_1 \sigma_2$ preserves the orientation, and $a_1+a_2\in \{0+0,1+1\}=\{0\}$. If exactly one of $\sigma_1$ and $\sigma_2$ preserves the orientation, then $\sigma_1 \sigma_2$ converts the orientation, and  $a_1+a_2=0+1=1$. Therefore $G'\subset G$ is a subgroup, and its index is 2. For each $(\sigma,a)\in G'$ we have that the corresponding diffeomorphism $g_\sigma\circ \pi^a$ preserves the orientation of ${\mathcal L}(\overline{{\mathcal M}_{0,n}^{\mathbb R}})$. Hence, $G'$  acts on ${\mathcal L}(\overline{{\mathcal M}_{0,n}^{\mathbb R}})$  and preserves the orientation. Let us show that $G'\cong S_n$. We define   the map $f:G'\to S_n $ by $f(\sigma,a)= \sigma $. Then $f$ is a bijection and preserves the group operation by construction.   Thus, we found the orientation preserving action of $S_n$ on ${\mathcal L}(\overline{{\mathcal M}_{0,n}^{\mathbb R}})$. 
	\end{proof} 
	
	\begin{re}  Further, speaking about the action of $S_n$ on ${\mathcal L}(\overline{{\mathcal M}_{0,n}^{\mathbb R}})$ we always assume the orientation preserving action introduced in Lemma~\ref{Sn}. 
	\end{re}

	\section{Cell decompositions of the varieties  $\overline{{\mathcal M}_{0,5}^{\mathbb R}}$ and ${\mathcal L}(\overline{{\mathcal M}_{0,5}^{\mathbb R}})$}
	
	\subsection{Cell decomposition of    $\overline{{\mathcal M}_{0,5}^{\mathbb R}}$}
	
	By Example \ref{ex1} the cell decomposition of $\overline{{\mathcal M}_{0,5}^{\mathbb R}}$ consists of 
	cells of the maximal dimension $\dim(\overline{{\mathcal M}_{0,5}^{\mathbb R}})=5-3=2$. Figure \ref{fig5_1} represents one of these cells. Namely, this is the cell  such that all curves in this cell are marked by the points  1, 2, 3, 4, 5 in this order. The boundary of this cell consists of the cells of dimension 1 (stable curves with one nodal point) and the cells of dimension 0 (stable curves with two nodal points). 
	We draw    1-dimensional cells as the edges of the pentagon at Figure \ref{fig5_1}  and  draw the typical representative of each cell nearby this edge.  
	
	\begin{figure}[h]
		\centering
		\epsfig{figure=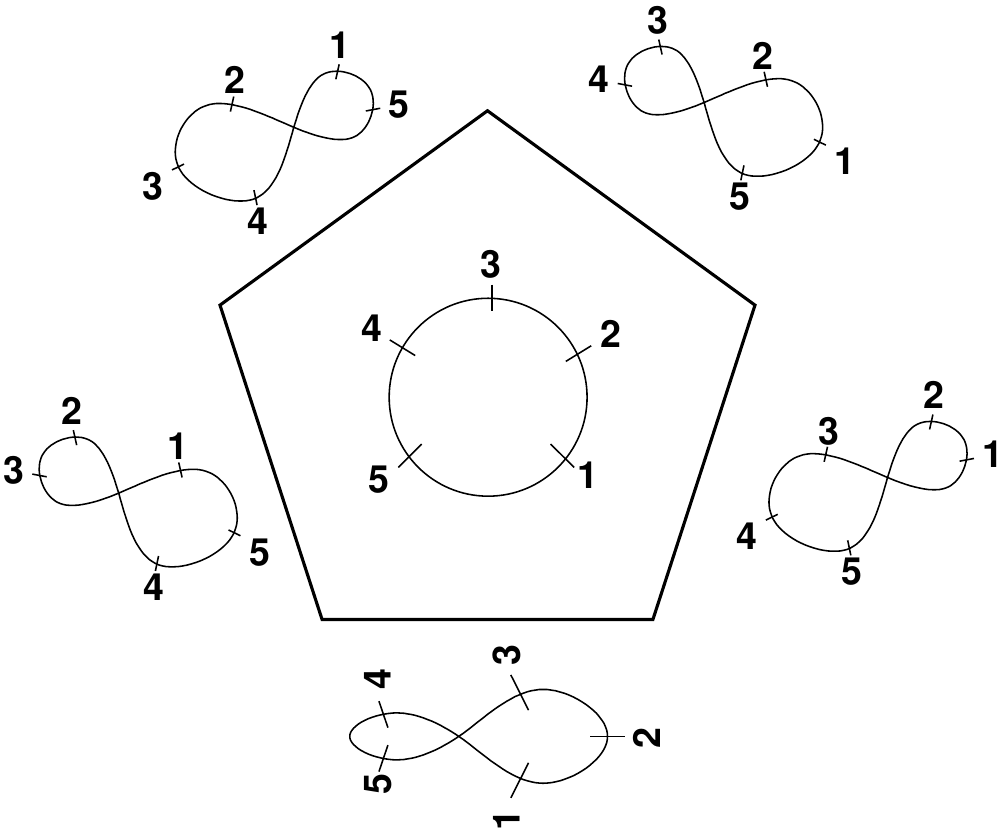,width=.35\linewidth}
		\caption{One of the dimension 2 cells of $\overline{{\mathcal M}_{0,5}^{\mathbb R}}$ with its boundary.} \label{fig5_1}
	\end{figure}
	
	There are 5 possibilities to split the ordered points  (1, 2, 3, 4, 5) into two components in such a way that a stable curve appears. Each 2-dimensional cell is a pentagon. 
	According to Construction~\ref{l:oMn_cell_dec} we mark the cell shown at Figure \ref{fig5_1} by the pentagon shown at Figure~\ref{Klgr0a}.

	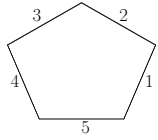
\begin{figure}[h]
		\centering
		\scalebox{0.4} 
		{
			\begin{picture}(130,100)
				\put(0,0){\line(1,0){80}}
				\qbezier{(-30,70)(-15,35)(0,0)}
				\qbezier{(110,70)(95,35)(80,0)}
				\qbezier{(-30,70)(5,90)(40,110)}
				\qbezier{(40,110)(75,90)(110,70)}
				\put(40,-13){{\LARGE $5$}}
				\put(100,30){{\LARGE $1$}}
				\put(76,93){{\LARGE $2$}}
				\put(-6,93){{\LARGE $3$}}
				\put(-27,30){{\LARGE $4$}}
			\end{picture}
		}
		\caption{The pentagon marking the cell corresponding to marked point order $(1, 2, 3, 4, 5)$.}
		\label{Klgr0a}
	\end{figure}

	Each of the dimension 1 boundary cells is marked by a pentagon with a single diagonal as it is shown at Figure \ref{Klgr0b}. The first one corresponds to the bottom edge of the cell at Figure \ref{fig5_1}, and then next in the contra-clockwise order are considered.

	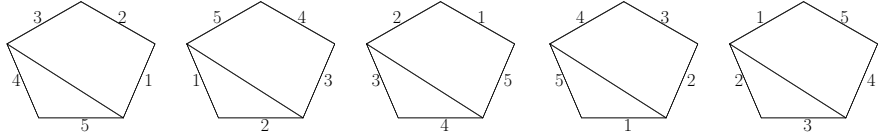
\begin{figure}[h]
		\centering
		\scalebox{0.4} 
		{
			\begin{picture}(510,120)
				\put(0,0){\line(1,0){80}}
				\put(170,0){\line(1,0){80}}
				\qbezier{(-30,70)(-15,35)(0,0)}
				\qbezier{(110,70)(95,35)(80,0)}
				\qbezier{(-30,70)(5,90)(40,110)}
				\qbezier{(40,110)(75,90)(110,70)}
				\qbezier{(140,70)(155,35)(170,0)}
				\qbezier{(280,70)(265,35)(250,0)}
				\qbezier{(140,70)(175,90)(210,110)}
				\qbezier{(210,110)(245,90)(280,70)}
				\qbezier{(140,70)(195,35)(250,0)}
				\qbezier{(-30,70)(25,35)(80,0)}
				\put(40,-12){{\LARGE $5$}}
				\put(100,30){{\LARGE $1$}}
				\put(75,90){{\LARGE $2$}}
				\put(-5,90){{\LARGE $3$}}
				\put(-25,30){{\LARGE $4$}}
				\put(210,-12){{\LARGE $2$}}
				\put(270,30){{\LARGE $3$}}
				\put(245,90){{\LARGE $4$}}
				\put(165,90){{\LARGE $5$}}
				\put(145,30){{\LARGE $1$}}
				
				\put(340,0){\line(1,0){80}}
				\qbezier{(310,70)(325,35)(340,0)}
				\qbezier{(450,70)(435,35)(420,0)}
				\qbezier{(310,70)(345,90)(380,110)}
				\qbezier{(380,110)(415,90)(450,70)}
				
				\qbezier{(310,70)(365,35)(420,0)}

				\put(380,-12){{\LARGE $4$}}
				\put(440,30){{\LARGE $5$}}
				\put(415,90){{\LARGE $1$}}
				\put(335,90){{\LARGE $2$}}
				\put(315,30){{\LARGE $3$}}

			\end{picture} 
				\begin{picture}(240,120)
					\put(0,0){\line(1,0){80}}
					\put(170,0){\line(1,0){80}}
					\qbezier{(-30,70)(-15,35)(0,0)}
					\qbezier{(110,70)(95,35)(80,0)}
					\qbezier{(-30,70)(5,90)(40,110)}
					\qbezier{(40,110)(75,90)(110,70)}
					\qbezier{(140,70)(155,35)(170,0)}
					\qbezier{(280,70)(265,35)(250,0)}
					\qbezier{(140,70)(175,90)(210,110)}
					\qbezier{(210,110)(245,90)(280,70)}
					\qbezier{(140,70)(195,35)(250,0)}
					\qbezier{(-30,70)(25,35)(80,0)}
					\put(40,-12){{\LARGE $1$}}
					\put(100,30){{\LARGE $2$}}
					\put(75,90){{\LARGE $3$}}
					\put(-5,90){{\LARGE $4$}}
					\put(-25,30){{\LARGE $5$}}
					\put(210,-12){{\LARGE $3$}}
					\put(270,30){{\LARGE $4$}}
					\put(245,90){{\LARGE $5$}}
					\put(165,90){{\LARGE $1$}}
					\put(145,30){{\LARGE $2$}}
			\end{picture}	} 
		\caption{The pentagons marking the boundary cells}
		\label{Klgr0b}
	\end{figure}

	The pentagons marking the two-dimensional cells on the other side of each dimension one boundary cell are shown at Figure \ref{Klgr0c}. Here each cell consists from the smooth curves. The order of the marked points on each curve inside a cell is given by the order of marks on the edges of the corresponding pentagon.

	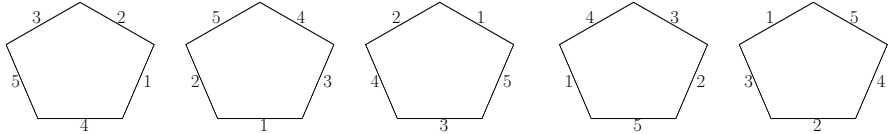
\begin{figure}[h]
		\centering
		\scalebox{0.4} 
		{
			\begin{picture}(520,120)
				\put(0,0){\line(1,0){80}}
				\put(170,0){\line(1,0){80}}
				\qbezier{(-30,70)(-15,35)(0,0)}
				\qbezier{(110,70)(95,35)(80,0)}
				\qbezier{(-30,70)(5,90)(40,110)}
				\qbezier{(40,110)(75,90)(110,70)}
				\qbezier{(140,70)(155,35)(170,0)}
				\qbezier{(280,70)(265,35)(250,0)}
				\qbezier{(140,70)(175,90)(210,110)}
				\qbezier{(210,110)(245,90)(280,70)}

				\put(40,-12){{\LARGE $4$}}
				\put(100,30){{\LARGE $1$}}
				\put(75,90){{\LARGE $2$}}
				\put(-5,90){{\LARGE $3$}}
				\put(-25,30){{\LARGE $5$}}
				\put(210,-12){{\LARGE $1$}}
				\put(270,30){{\LARGE $3$}}
				\put(245,90){{\LARGE $4$}}
				\put(165,90){{\LARGE $5$}}
				\put(145,30){{\LARGE $2$}}
				
				\put(340,0){\line(1,0){80}}
				\qbezier{(310,70)(325,35)(340,0)}
				\qbezier{(450,70)(435,35)(420,0)}
				\qbezier{(310,70)(345,90)(380,110)}
				\qbezier{(380,110)(415,90)(450,70)}

				\put(380,-12){{\LARGE $3$}}
				\put(440,30){{\LARGE $5$}}
				\put(415,90){{\LARGE $1$}}
				\put(335,90){{\LARGE $2$}}
				\put(315,30){{\LARGE $4$}}

			\end{picture} 
			\begin{picture}(260,120)
				\put(0,0){\line(1,0){80}}
				\put(170,0){\line(1,0){80}}
				\qbezier{(-30,70)(-15,35)(0,0)}
				\qbezier{(110,70)(95,35)(80,0)}
				\qbezier{(-30,70)(5,90)(40,110)}
				\qbezier{(40,110)(75,90)(110,70)}
				\qbezier{(140,70)(155,35)(170,0)}
				\qbezier{(280,70)(265,35)(250,0)}
				\qbezier{(140,70)(175,90)(210,110)}
				\qbezier{(210,110)(245,90)(280,70)}
				\put(40,-12){{\LARGE $5$}}
				\put(100,30){{\LARGE $2$}}
				\put(75,90){{\LARGE $3$}}
				\put(-5,90){{\LARGE $4$}}
				\put(-25,30){{\LARGE $1$}}
				\put(210,-12){{\LARGE $2$}}
				\put(270,30){{\LARGE $4$}}
				\put(245,90){{\LARGE $5$}}
				\put(165,90){{\LARGE $1$}}
				\put(145,30){{\LARGE $3$}}
		\end{picture}	}
		\caption{The pentagons marking the two-dimensional cells neighbor to the cell   at Figure~\ref{Klgr0a}}
		\label{Klgr0c}
	\end{figure}

	Two neighbor cells are shown at Figure  \ref{fig5_2}. The upper one is marked by the pentagon   at Figure \ref{Klgr0a}, the bottom one is marked by the first of the pentagons   at Figure \ref{Klgr0c}, and the boundary between them  is marked by the first pentagon   at Figure~\ref{Klgr0b}.

	\begin{figure}[h] 
		\centering
		\epsfig{figure=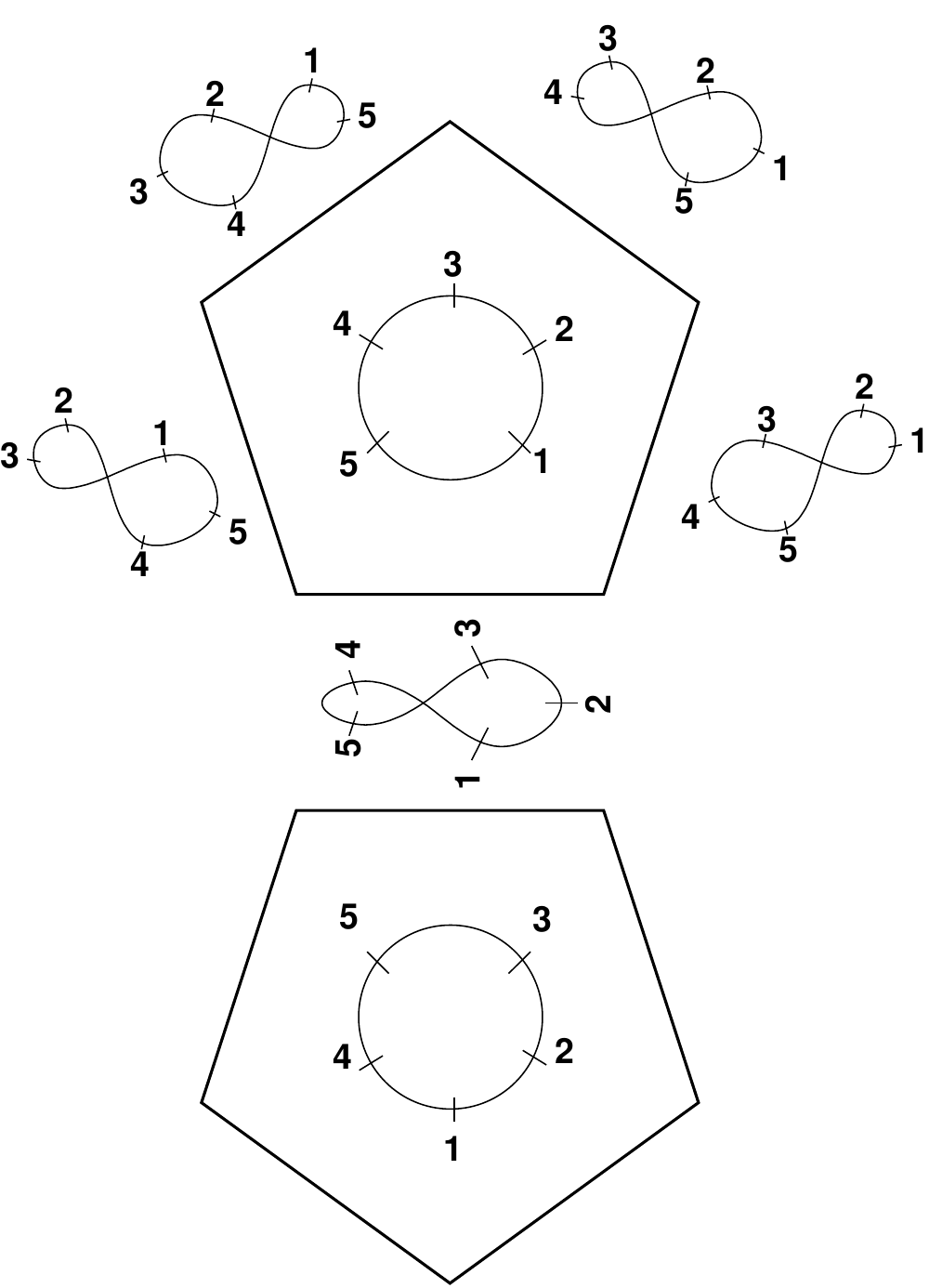,width=.30\linewidth}
		\caption{Two neighbor  cells of $\overline{{\mathcal M}_{0,5}^{\mathbb R}}$.} \label{fig5_2}
	\end{figure}

	\begin{lemma} \label{LemNum}
		The  cell decomposition of $\overline{{\mathcal M}_{0,5}^{\mathbb R}}$ contains 12 cells of dimension 2, 30 cells of dimension 1 and 15 cells of dimension 0.
	\end{lemma}
	\begin{proof}
		The number of the cells of dimension 2 is equal to the number of the orderings of 1, 2, 3, 4, 5, divided by the order of the dihedral group of the pentagon. Hence, it is $\frac{5!}{5\cdot 2}=12$. 	
		
		The number of the cells of codimension 1  is equal to ${5\choose 2}\cdot {3\choose 1}=30$ since we have to choose 2 points on a separate component and 1 middle point from the other 3 points.
		
		The cells of codimension 2 consist of the 3-component stable curves. The component, which has 2 common points with the other components, contains 1 marked point and is called middle. Two other components contain 2 marked points on each of them.  The number of these cells is equal to $\frac12\cdot 5\cdot {4\choose 2} =15$ since we have to choose 1 of 5 points on the middle component, then 2 points on one of the  rest components, and then divide by 2 since these two components cannot be distinguished.
		
		Gluing these  cells we get the decomposition of $\overline{{\mathcal M}_{0,5}^{\mathbb R}}$ to  the pentagons, as it is shown on Figure~\ref{fig5_3}. 
	\end{proof}
	
	\begin{re} 
		The variety $\overline{{\mathcal M}_{0,5}^{\mathbb R}}$ is depicted at the Figures \ref{fig5_3} and \ref{fig5_4}. Gluing  between all 12 cells is shown at Figure~\ref{fig5_3}.
		At Figure  \ref{fig5_4} we draw the same surface, where all 
		sides marked by $a_1,\ldots, a_5$ are glued, and the sides marked by $b_1,b_2,b_3=c_1\cup c_2\cup c_3$ should be glued as the arrows show.
	\end{re}
	
	
	\begin{figure}[!h] 
		\centering{\epsfig{keepaspectratio,trim=110pt 150pt 0pt 170pt,clip,figure=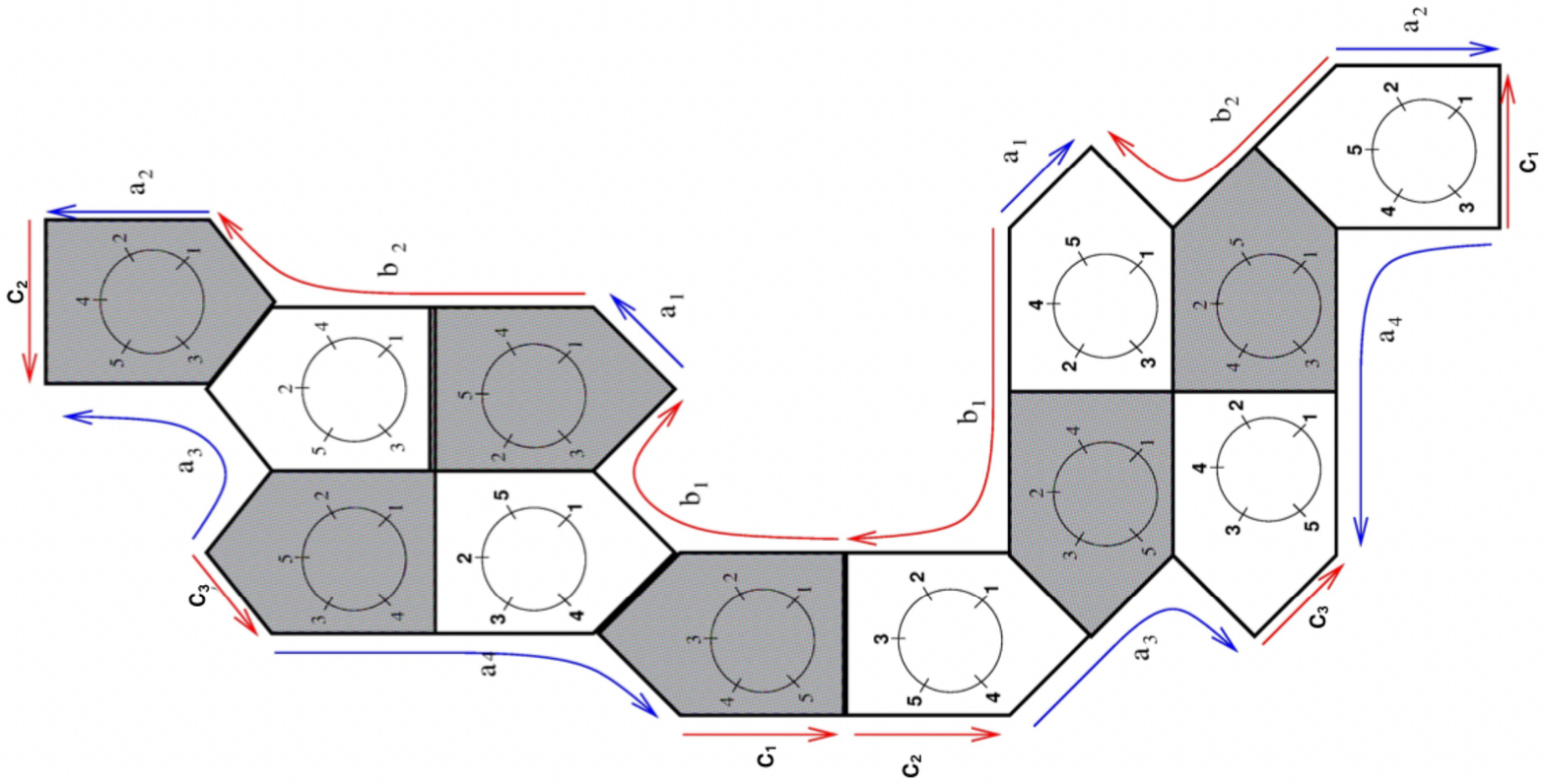,width=1.2\linewidth}} 
		\caption{Cell decomposition of the variety $\overline{{\mathcal M}_{0,5}^{\mathbb R}}$} \label{fig5_3}
	\end{figure}

	\begin{figure}[!h]
		\centering\epsfig{figure=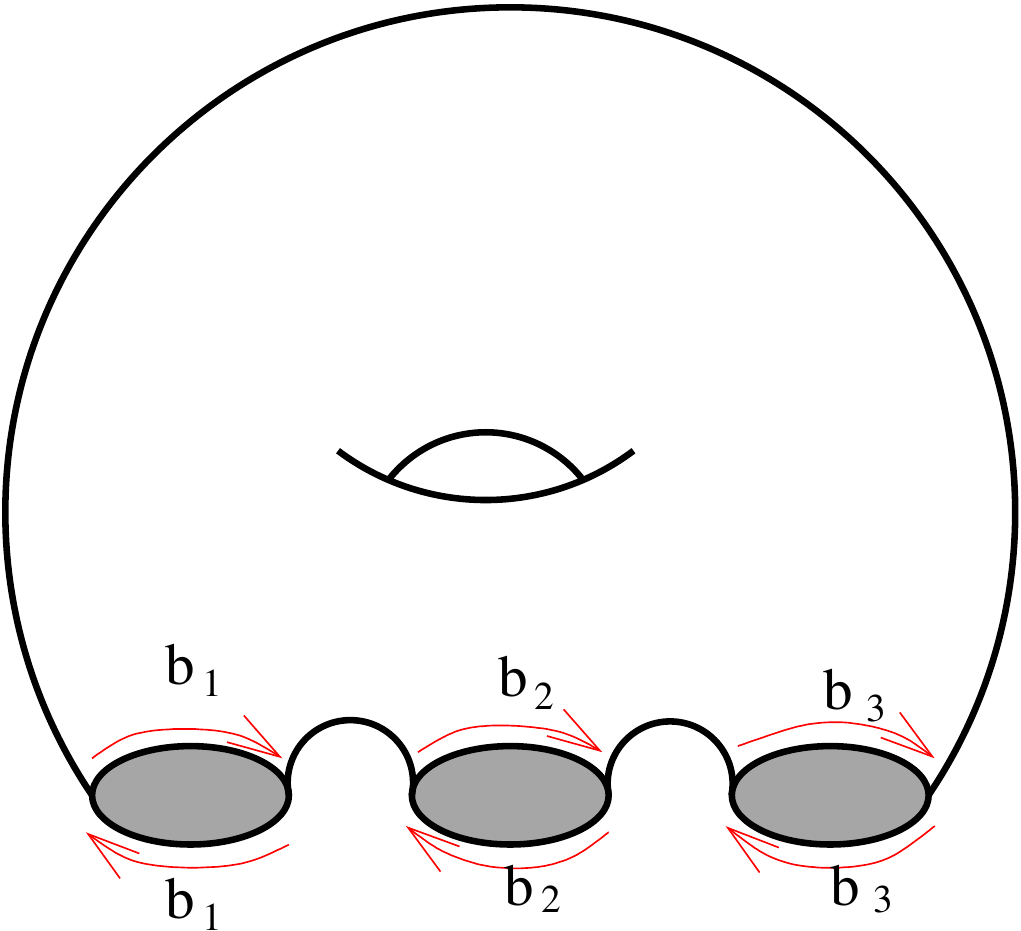,width=0.40\linewidth}
		\caption{Gluing of $\overline{{\mathcal M}_{0,5}^{\mathbb R}}$}  \label{fig5_4}
	\end{figure}

	\subsection{Cell decompositions of   ${\mathcal L}(\overline{{\mathcal M}_{0,5}^{\mathbb R}})$}
	
	

	By~\cite[Theorem 1.2, Figure 6]{AmbKre}  the orientation cover ${\mathcal L}(\overline{{\mathcal M}_{0,5}^{\mathbb R}})$  
	given by  Definition \ref{dOC} can be obtained in the following way. 
	Each cell of the maximal dimension described in Construction \ref{l:oMn_cell_dec} is   doubled with the both possible orientations. Then all cells are   glued altogether in such a way that if $C_1$ and $C_2$ are the cells glued in $\overline{{\mathcal M}_{0,5}^{\mathbb R}}$, then $C_1$ is glued   either to $C_2$ or to its doubling $C_2'$ such that  the orientation of cells according to common edges is consistent. Correspondingly, $C_1'$ is glued either to $C_2'$ or to $C_2$.  
	By~\cite[Propositions 2.1 and 2.2]{Kreck} the gluing of all 24 cells altogether is a connected orientable surface and  the result of gluing  is the orientation cover ${\mathcal L}(\overline{{\mathcal M}_{0,5}^{\mathbb R}})$ with a fixed orientation. 
	%
	%
	Euler formula gives that the genus of this surface is $4$ 
	and it can be  obtained from Figure \ref{fig5_4} by cutting along the edges $b_1,b_2,b_3$, marking one of the two obtained edges $b_i$   by   $b_i'$, applying mirror symmetry to the obtained surface with boundary $b_1\cup b_1'\cup b_2\cup b_2'\cup b_3\cup b_3'$ and gluing $b_i$ on the one surface with a boundary with $b_i'$ on another surface with a boundary, $i=1,2,3$.
	
	The exact way how we glue    $24=12\cdot 2$ pentagons to obtain ${\mathcal L}(\overline{{\mathcal M}_{0,5}^{\mathbb R}})$  is shown at Figure~\ref{fig5_5}. Another way to scan and visualize this variety can be seen in \cite[Fig.~13]{DevMor}.  
	
	\begin{figure}[!h]
		\centering\epsfig{figure=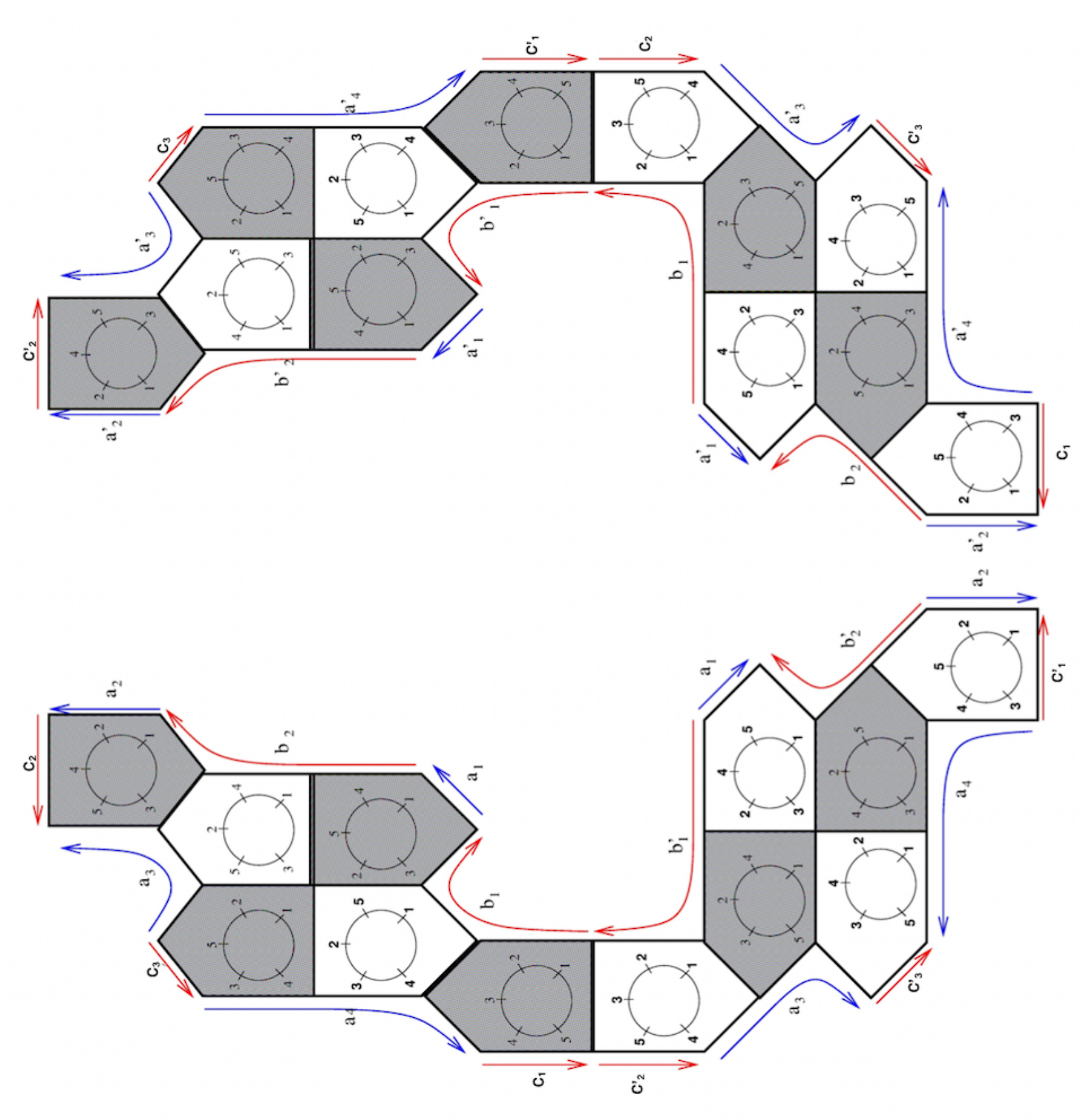,width=.90\linewidth}
		\caption{The scanning of ${\mathcal L}(\overline{{\mathcal M}_{0,5}^{\mathbb R}})$}  \label{fig5_5}
	\end{figure}

	\section{Dessins d'enfants and Belyi pairs}
	
	\subsection{Basic definitions and notations}
	\begin{de} \label{DefDD}
		A \emph{dessin d'enfant} is a~compact connected smooth oriented surface~$S$
		together with a bi-colored graph~$\Gamma$ embedded into $S$, such that the complement
		$S\setminus \Gamma$ is homeomorphic to a~disjoint union of open disks.
	\end{de}

	\begin{de}  \label{DefBP}
		Let $\mathcal{X}$ be an irreducible smooth algebraic curve over $\mathbb{C}$. A function $\beta$ on $\mathcal{X}$ is called a \textit{Belyi function}, if it defines the covering $\beta : \mathcal{X} \rightarrow \mathbb{P}^1(\mathbb{C})$ which is unramified  over all points from  $\mathbb{P}^1(\mathbb{C}) \setminus \{0,1,\infty\}$. The pair $(\mathcal{X},\beta)$ is called a \textit{Belyi pair}.
	\end{de}
	
	Dessins d'enfants are  naturally related with Belyi pairs   in the following way.   A topological model of~$\mathcal{X}$ is used to be  the surface $S$. For the graph we consider the preimage
	$\beta^{-1}([0,1])$.  
	Here the edges  are $\{\beta ^{-1}((0,1))\}$, white vertices are
	$\{\beta ^{-1 }(\{1\})\}$, and black vertices are $\{\beta ^{-1 }(\{0\})\}$. It is known that since  $(\mathcal{X},\beta)$ is  a Belyi pair, a dessin d'enfant appears by this construction. Also for any dessin d'enfant there is a Belyi pair to appear from. Moreover, it is proved in \cite{ShabVoev} that for naturally defined morphisms the category of Belyi pairs is equivalent to the category of dessins d'enfants.    
	Detailed and self-contained information concerning the correspondence between   dessins and Belyi pairs arising from this equivalence can be found in \cite{AmbDis,AmbKreSh,ShabNew,ShabVoev}. We present just several known properties that we are going to use below.
	
	\begin{re}  Let $(S,\Gamma)$ be a dessin d'enfant and $(\mathcal{X},\beta)$ be its Belyi pair.
		
		1. Converting the vertex coloring in $\Gamma$ corresponds to the following transformation of the  Belyi pairs: $(\mathcal{X},\beta) \to (\mathcal{X}, 1-\beta)$. 
		
		2. The transformation $(\mathcal{X},\beta) \to (\mathcal{X}, 
		{4\beta(1-\beta)})$ converts   $\Gamma$ to the graph $\Gamma'$ embedded into the same surface $S$. The  black vertices of $\Gamma'$ are all vertices of $\Gamma$  and  white vertices  of $\Gamma'$ are of valency 2 and are located at the middles of all edges of $\Gamma$.  
		
		3.  The transformation $(\mathcal{X},\beta) \to (\mathcal{X}, \displaystyle \frac{1}
		{\beta})$ converts   $\Gamma$ to the graph $\Gamma''$ embedded into the same surface $S$. The white vertices of $\Gamma''$ and $\Gamma$ are the same, but centers of faces of $\Gamma$ became black vertices of $\Gamma''$ and vice versa. The edges of $\Gamma''$ are the preimages of the interval $(0,1)$   of the function $\frac{1}
		{\beta}$, i.e., the preimages of the interval $(1,\infty)$  of the function~$ \beta$.
		\label{Re4.3}
	\end{re}

	\begin{de} 
		A dessin d'enfant is called \emph{regular} if it has an edge–transitive, colour– and orientation–
		preserving automorphism group. \end{de}

	\subsection{Dessin d'enfant for ${\mathcal L}(\overline{{\mathcal M}_{0,5}^{\mathbb R}})$}

	\begin{lemma}
		Drawing white vertices in the middles of edges  of the cell decomposition of ${\mathcal L}(\overline{{\mathcal M}_{0,5}^{\mathbb R}})$ provides  a   dessin d'enfant ${\mathcal D}$ on the genus 4 surface. The faces of ${\mathcal D}$ are  10-gons. ${\mathcal D}$ contains 24 faces, 120 edges, 30 black  vertices of valency 4, and  60 white  vertices of valency 2. 
	\end{lemma}
	
	\begin{proof}
		The result follows from Lemma~\ref{LemNum} and Definition~\ref{dOC}. 
	\end{proof}

	\subsection{Dual dessins d'enfants} 
	The concept of dual dessins d'enfants arises while someone interchanges the critical values of a Belyi function. Correspondingly, 6 types of duality, or, probably, 3 types of duality and 2 types of 'thriality' arise. However, at  the moment only the case of so-called clean Belyi pairs, where each ramification over 1 has the order 2 is  well-known and investigated. In this case dual dessin is  the preimage of $[0,1]$ segment of the  function $1/\beta$. The detailed and self-contained information on this subject can be found in~\cite{GurShab}.
	
	We generalize the notion introduced in~\cite{GurShab} to the general case verbatim.
	
	\begin{de} 
		Let $(X,\Gamma)$ be a dessin d'enfant, and $(\mathcal{X},\beta)$ be the corresponding Belyi pair. A dessin $(X,\Gamma^*)$  is called \emph{dual} to $(X,\Gamma )$ if $\Gamma^*$ is the preimage of $[0,1]$ of the  function $1/\beta:\mathcal{X} \rightarrow \mathbb{P}^1(\mathbb{C}).$ 
	\end{de} 
	
	\begin{re} 
		A dessin $(X,\Gamma^*)$  is   dual to $(X,\Gamma)$ iff $(X,\Gamma^*)$ is a dessin d'enfant such that the sets of white vertices of $\Gamma$ and $\Gamma^*$ coincide, the set of black vertices of $\Gamma^*$ coincides with the set of centers of faces of~$\Gamma$, and the edges connect centers of faces of~$\Gamma$ with all white vetrices incident to this face.
	\end{re} 
	
	\begin{de}  \label{DeDu}
		Let $(X,\Gamma)$ be a dessin d'enfant, $(X,\Gamma^*)$ be its dual dessin. The union   $(X,\Gamma)\cup (X,\Gamma^*)$ is the dessin on $X$ defined as follows: the white vertices are the common white verices of the graphs $\Gamma$ and $\Gamma^*$, its black vertices are constituted by the union of black vertices of $\Gamma$ and black vertices of $\Gamma^*$, and its edges are the union of edges of $\Gamma$ and~$\Gamma^*$. 
	\end{de}

	\begin{re}
		Let $(X,\Gamma)$ be a dessin d'enfant, and $\beta$ be its Belyi function. Then   
		the Belyi function of the union $(X,\Gamma)\cup (X,\Gamma^*)$ is $\frac{4\beta}{(\beta+1)^2}$ on the same curve, this can be verified directly, see also \cite[Example~5.1]{Zv}.
	\end{re}

	\section{ 4-icosahedron and its properties }
	
	In this section we follow the paper \cite{Zv} by A. Zvonkin, where among the other results  the Belyi pair for  4-icosahedron is investigated.
	
	\subsection{Bring curve}
	
	\begin{de}
		\emph{Bring curve} is an algebraic curve in 4-dimensional complex projective space with coordinates $x_1:\ldots :x_5$ defined by the system of equations:
		$$B_5:\quad \left\{ \begin{array}{l} \sum_{i=1}^5 x_i=0\\ \sum_{i=1}^5 x^2_i=0\\\sum_{i=1}^5 x^3_i=0\end{array}\right.$$
	\end{de}
	
	In 1786 Erland Bring, a history professor at
	Lund, found a change of variables which reduces a generic quintic equation to the   form $ q(x)= x^5 + a x + b$. Then in 1884 Felix Klein \cite{Klein}  introduced Bring curve and firstly investigated its properties. A set of five roots of the equation $q(x)=0$ usually gives rise to 120 points on $B_5$ which are different permutations of these  roots. For some particular values of parameters $a$ and $b$ the number of such points may become smaller.

	\begin{lemma} \cite[Example 5.5, Proposition 5.6]{Zv} Let $x_1,\ldots ,x_5$ be the zeros of the equation $x^5+ax+b=0$. Then the function $$f_{B_5}(x_1:\ldots :x_5)=\frac{256a^5}{256 a^5+3125 b^4}$$ is a Belyi function on the Bring curve~$B_5$. The degree of $f_{B_5}$ on $B_5$ is~$120$.
	\end{lemma}

	\begin{re}
		By classical Vieta theorem, the coefficients of the equation $q(x)=0$ are the elementary symmetric  functions in the roots $x_1,\ldots, x_5$ of  $q(x)=0$. Since $q(x)=x^5+ax+b$ we have that the elementary symmetric functions of the degrees 1, 2, 3 of $x_1,\ldots, x_5$
		are 0. Therefore, the power sums of the degrees 1, 2, 3 of $x_1,\ldots, x_5$ are 0. Thus  $(x_1:\ldots: x_5)$ satisfies the equations of the Bring curve $B_5$, i.e., $(x_1:\ldots: x_5)\in B_5$. The coefficients $a$ and $b$  are the symmetric polynomials in $x_1,\ldots, x_5$ of degrees
		4 and 5, respectively. Since  $(x_1 : x_2 : x_3 : x_4 : x_5)$ is
		a projective point, the pairs $(a, b)$ should be considered up to the equivalence relation
		$(a, b) \sim (\lambda^
		4a, \lambda^5b)$ for any $\lambda\in {\mathbb C}\setminus \{0\}$. Then the expression  $ \frac{256a^5}{256 a^5+3125 b^4}$ depends only on the point  $(x_1 : x_2 : x_3 : x_4 : x_5)$ on the Bring curve, i.e., it is a   function on~$B_5$.
	\end{re}
	
	\subsection{Another icosahedron}
	
	We start with 
	a regular icosahedron.  The icosahedron is a Platonic solid consisting of 20 triangles. Each vertex of this solid is incident to 5 different triangles. It has 30 edges and 12 vertices. This solid produces a regular dessin d'enfant on a sphere by adding a white vertex to the middle of each of the edges. So, we have a dessin d'enfant of genus 0 with 30 white vertices, 12 black vertices that are the vertices of the icosahedron, 60 edges, and 20 faces.

	While a graph of an icosahedron is embedded into the oriented surface, we have a cyclic order of edges in any vertex, so the   relation  "to be next one" on edges in a vertex can be defined. Below we want to fix the names for the   edges in any vertex. 
	\begin{de} \label{de:neigh} For any black vertex $v$ let us enumerate the edges incident to $v$ in the contra-clockwise order as $(e_1^v,e_2^v,e_3^v,e_4^v,e_5^v)$. Here we can start with any fixed edge.
	\end{de}
	We fix the introduced edge marks for the future. Each edge has exactly one mark since it has one black vertex and one white vertex.
	
	Let us consider a graph of the icosahedron as an abstract graph now. For any abstract graph, if we determine a cyclic order of edges in all vertices, we embed this graph into a certain surface. This procedure creates a dessin d'enfant from this graph, see \cite[Section 1.3.3]{LZ}.  Below  a regular dessin of positive genus is provided by  the icosahedron graph, cf.~\cite[Exercise 1.3.13]{LZ}.

	\begin{de} \label{defI_4}
		Let us consider a dessin d'enfant obtained from the icosahedron graph by fixing the $(e_1^v,e_3^v,e_5^v,e_2^v,e_4^v)$ order of the edges in each black vertex $v$. 
		Following \cite{Zv} we denote this dessin by $I_4$ and call it 4-icosahedron. \end{de} 
	
	Each face of such dessin is a pentagon. In a given vertex $v$ this pentagon can be obtained as a section of the original Platonic solid by a plane containing a pair of edges from the set $\{(e_1^v,e_3^v)$, $(e_2^v,e_4^v)$,  $(e_3^v,e_5^v)$, $(e_4^v,e_1^v)$,   $(e_5^v,e_2^v)\}$. So, 5 faces are intersecting in the vertex~$v$.
	
	\begin{re}
		If $\sigma$ is a permutation determining cycle order of edges in black vertices of regular icosahedron then $I_4$ is determined by~$\sigma^2$.
	\end{re}
	
	Note that the valency of any white vertex is 2 by construction. So, there is no possibility to change the order of edges in white vertices.
	
	For the introduced dessin $I_4$ the following statement holds.
	
	\begin{lemma} \cite[Example 2.9]{Zv} $I_4$ is a regular dessin d'enfant on a genus 4 surface. It is a figure with 12 black vertices of valency 5, 30 white vertices of valency 2, 60 edges 
		and 12 faces of valency 5. Its automorphism group is~$A_5$.  \label{LemI4} \end{lemma}
	
	Figure \ref{F8a} is the scanning of the dessin $I_4$. Here equal symbols mark the edges that have to be glued to obtain the original surface. Note that some of the authomorphisms of the dessin can be easily seen on this scanning. In particular, the rotation of different parts in opposite directions on the angle $2\pi/5$ is an authomorphism. The map $p_i \leftrightarrow q_i$ and $s_i \leftrightarrow r_i$, $i=1,\ldots, 5$ is a mirror symmetry (not automorphism) of the dessin.
	
	
	\begin{figure}[htbp]
		\includegraphics[trim=0pt 50pt 0pt 230pt,clip,width=1\textwidth]{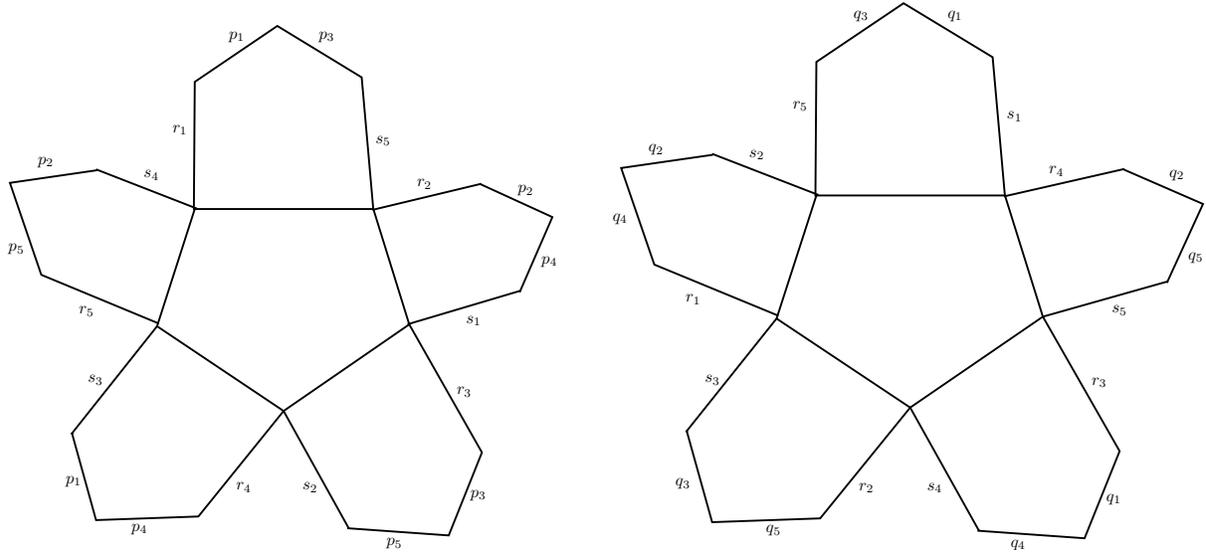}
		\caption{Scanning of $I_4$} \label{F8a}
	\end{figure}
	
	Beautiful pictures of the dessin $I_4$, including some mosaic on the floor in St.Mark’s basilica, Venice; attributed to Paolo Ucello
	(around 1430), can be found in \cite[pages 338 and~368]{Zv}. 
	
	Let $I_4^*$ denotes the dual dessin d'enfant for~$I_4$. 
	
	\begin{lemma} \cite[Example 5.5]{Zv} The dessins $I_4$ and $I^*_4$ are isomorphic.  \end{lemma}

	We consider the union $I_4\cup I_4^*$ as a  dessin d'enfant determined by Definition~\ref{DeDu}. 
	This dessin has 24 black vertices of the valency 5, 30 white vertices of the valency 4, and 60 faces, each of which is a quadrilateral.
	
	\begin{theorem} \cite[Proposition 5.6]{Zv} 1. Automorphism group of the dessin $I_4 \cup I_4^*$ is the symmetric group $S_5$.
		
		2. The Belyi pair corresponding to $I_4 \cup I_4^*$ is $(B_5,f_{B_5})$.
	\end{theorem}

	\section{Belyi pair of  the cell decomposition of   ${\mathcal L}(\overline{{\mathcal M}_{0,5}^{\mathbb R}})$}

	\begin{definition} \label{Def:J}
		We denote by $\mathcal J$ the  dessin obtained from $(I_4\cup I_4^*)^*$ by re-coloring of white and black vertices. 
	\end{definition}
	
	\begin{lemma} \label{L0}
		The Belyi pair of   ${\mathcal J}$    is $(B_5,1-\frac1{f_{B_5}})$.
	\end{lemma}

	\begin{proof} We remark that the transformation $\rho:  x\to 1-1/x$ maps the interval $[1,\infty)$ to the segment $[0,1]$ moving $1 \to 0$ and $\infty \to 1$. The rest follows from the definition of ${\mathcal J}$. 
	\end{proof}
	
	The following lemma describes the structure of the dessin~${\mathcal J}$. 
	
	\begin{figure}[!ht]
		\includegraphics[trim=0pt 50pt 0pt 230pt,clip,width=1\textwidth]{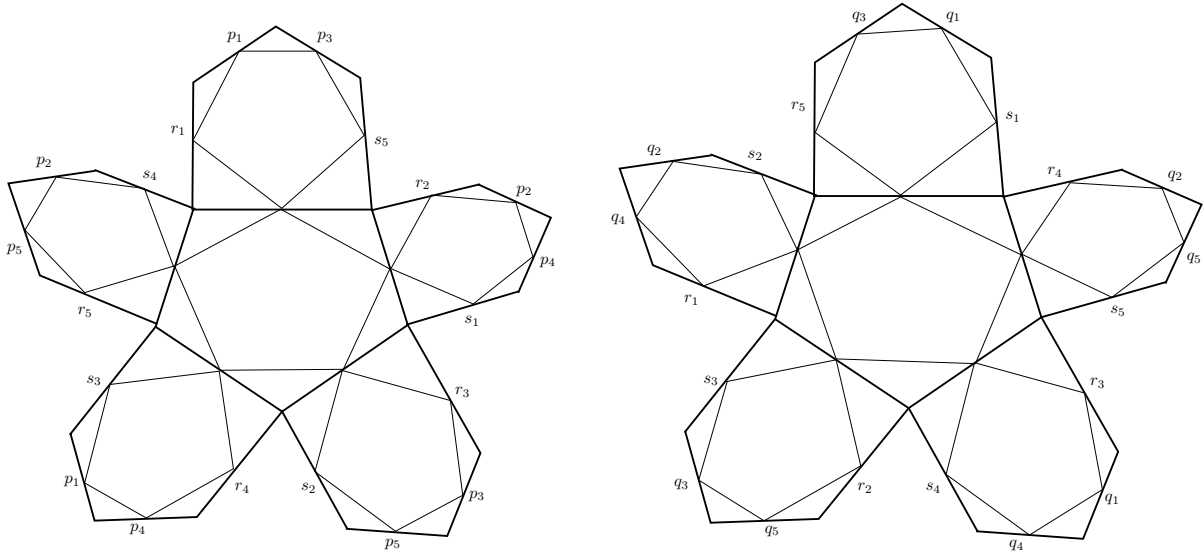} 
		\caption{Scanning of the dessins $I_4$ (dark lines) and $\mathcal J$ (light lines) from Definition~\ref{Def:J}} 	\label{F9d}
	\end{figure}
	
	\begin{lemma} \label{l1} 
		1. All white vertices of $\mathcal J$ are of valency 2.



		2. The dessin ${\mathcal J}$ without white vertices is shown at Figure \ref{F9d} by the light lines. It is embedded into a surface obtained from the figure drawn by the dark lines by gluing the equally marked sides. 
		
		3. The faces of the dessin   $\mathcal J$ are 24 bi-colored 10-gons. 
	\end{lemma}

	\begin{proof} We start from a face of $I_4$ and transform it geometrically. 
		
		Below we consider one of the pentagons shown at Figure~\ref{F8a}, cf. Figure \ref{F8b}(a).  We add white vertices at the middles of edges to obtain  Figure \ref{F8b}(b) and draw the   dual dessin $I_4^*$ in bold lines, see Figure \ref{F8b}(c). Then we draw the union of the pentagon and   the part of its dual dessin located inside this pentagon as a unique bi-colored dessin, Figure \ref{F8b}(d).   Figure \ref{F8b}(e) presents the dual dessin to the dessin from Figure \ref{F8b}(d). This  dual dessin  is drawn   in light lines. Note that all black vertices of the dual dessin from Figure \ref{F8b}(d) have valency 2, since the corresponding faces are quadrangles. So, below we can and do omit the vertices of valency 2 of the dual dessin from Figure \ref{F8b}(d). Changing the color of vertices and erasing the edges of the dessin $I_4^*$ in order to simplify the picture, we obtain the graph drawn at Figure \ref{F8b}(f). Repeating this procedure with each face of $I_4$ we obtain the graph at Figure~\ref{F9d}, i.e.,~${\mathcal J}$.

		
		\begin{figure}[!ht]
			\includegraphics[trim=0pt 80pt 0pt 420pt,clip,width=1\textwidth]{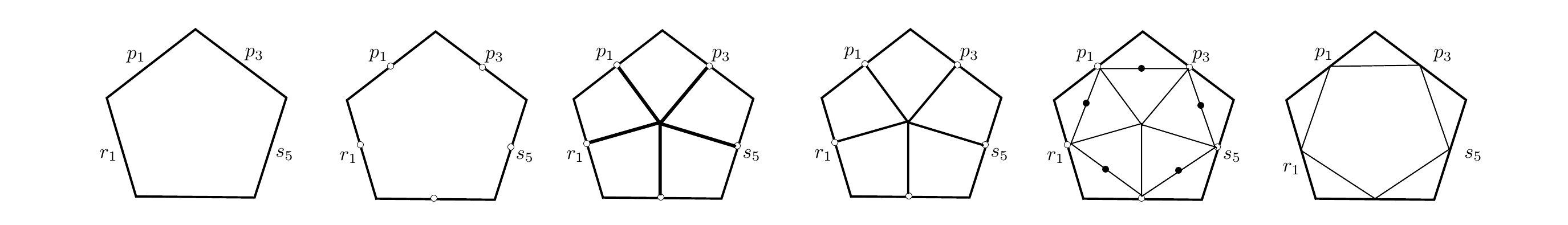} 
			\caption{Transition from a face of $I_4$ to a face of  $\mathcal J$} \label{F8b}
		\end{figure} 
	\end{proof}

	The main goal of this section is to show that the dessin $\mathcal J$   is isomorphic to the dessin $\mathcal D$ shown at Figure \ref{fig5_5}. We are going to do this within a series of lemmas step by step showing how we plan to cut and glue the scanning of $\mathcal J$ to obtain the required figure.

	\begin{lemma}
		Figure~\ref{F11} represents the scanning of the dessin $\mathcal J$. 
		\label{l2}
	\end{lemma}
	\begin{proof} 
		We show how we cut and glue some pieces of the dessin at Figure \ref{F9d} in order to obtain the picture at Figure~\ref{F11}. We cut three hatched triangles shown at Figure \ref{F9d1} from each of 10 outside pentagons and glue them by the lines marked by the same symbols. We denote the obtained new cuts by the symbols $x_i,y_i,u_i, v_i,z_i,t_i$, $i=1,2,3,4,5$, as it is shown at Figure~\ref{F9d1}. Note that there is a mirror symmetry  $x_i \leftrightarrow y_i$, $v_i \leftrightarrow z_i$, and $u_i \leftrightarrow t_i$.

		
		\begin{figure}[!h]
			\includegraphics[keepaspectratio=true, trim=0pt 0pt 0pt 230pt,clip,width=1\textwidth]{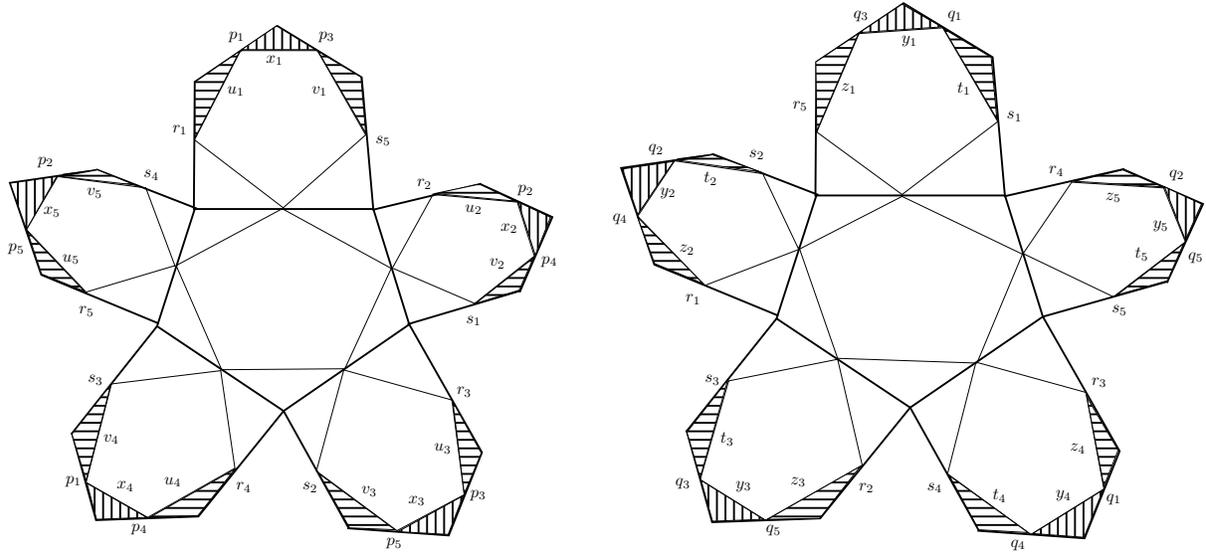} 
			\caption{The triangles for cutting off at the dessin $\mathcal J$} \label{F9d1}
		\end{figure}

		Then the horizontally hatched parts on the left hand side of  Figure~\ref{F9d1} should be   glued to became the  horizontally hatched parts  on the right hand side of    Figure~\ref{F9e} and vice versa. The vertically hatched parts should be   glued to became the pentagons at Figure~\ref{F11b}.

		\begin{figure}[!h]
			\includegraphics[trim=0pt 0pt 0pt 265pt,clip,width=1\textwidth]{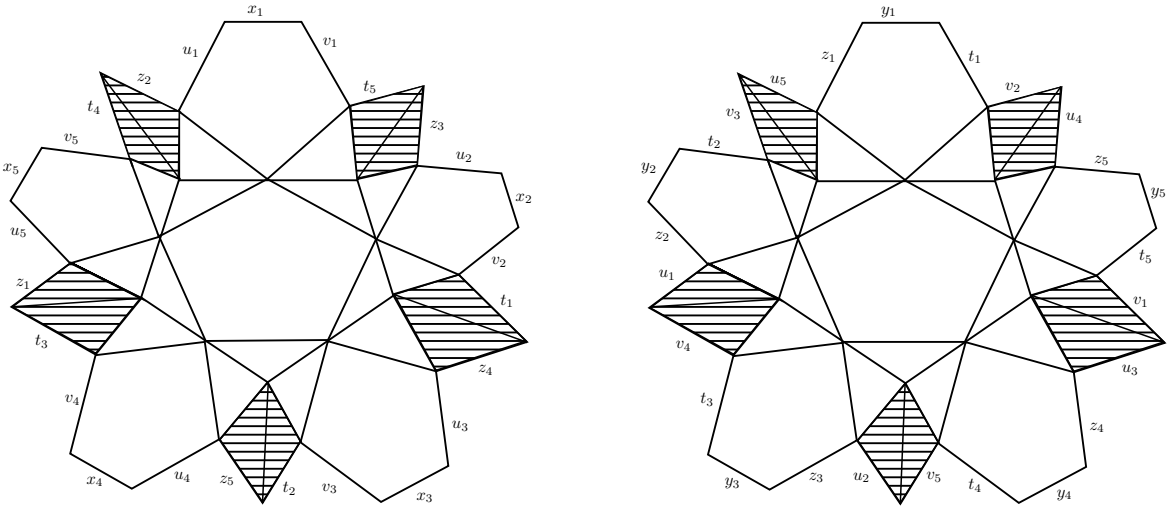} 
			\caption{Gluing of horizontally hatched parts of $\mathcal J$} 	\label{F9e}
		\end{figure}
		
		In particular, the triangle marked by $r_1, p_1, u_1$ on the left hand side of Figure~\ref{F9d1} became the upper half of the crossed spine marked by $u_1$ and $v_4$ on the right hand side of Figure~\ref{F9e}. Correspondingly, the lower half of this spine arises from the triangle marked by $p_1,s_3,v_4$  on the left hand side of Figure~\ref{F9d1}. In the spine these two triangles are glued by the edge marked with~$p_1$. 
		
		\begin{figure}[!h]
			\centering
			\includegraphics[trim=0pt 150pt 0pt 180pt,clip,width=0.6\textwidth]{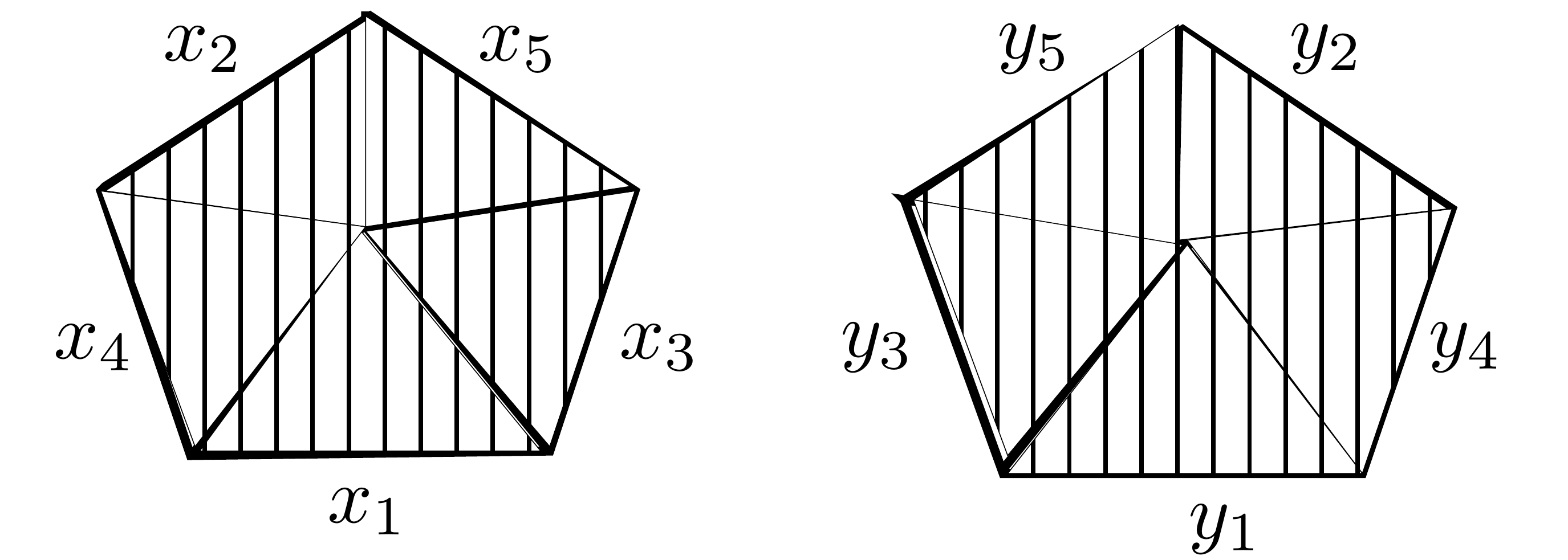} 
			\caption{Gluing of vertically hatched parts of $\mathcal J$} 	\label{F11b}
		\end{figure}
		
		The vertically hatched parts on the left hand side part of Figure~\ref{F9d1} should be cut of by the edges marked by $x_i$, $i=1,\ldots, 5$, and glued by the edges marked by $p_i$, $i=1,\ldots, 5$,    in order to obtain the left hand side pentagon at Figure~\ref{F11b}. Similarly, we cut vertically hatched triangles on the right hand side part of Figure~\ref{F9e} by the edges marked by $y_i$, $i=1,\ldots, 5$, and glued by the edges marked by $q_i$, $i=1,\ldots, 5$,  in order to obtain the right hand side pentagon at Figure~\ref{F11b}.
		
		\begin{figure}[htbp]
			\includegraphics[trim=0pt 80pt 0pt 100pt,clip,width=1\textwidth]{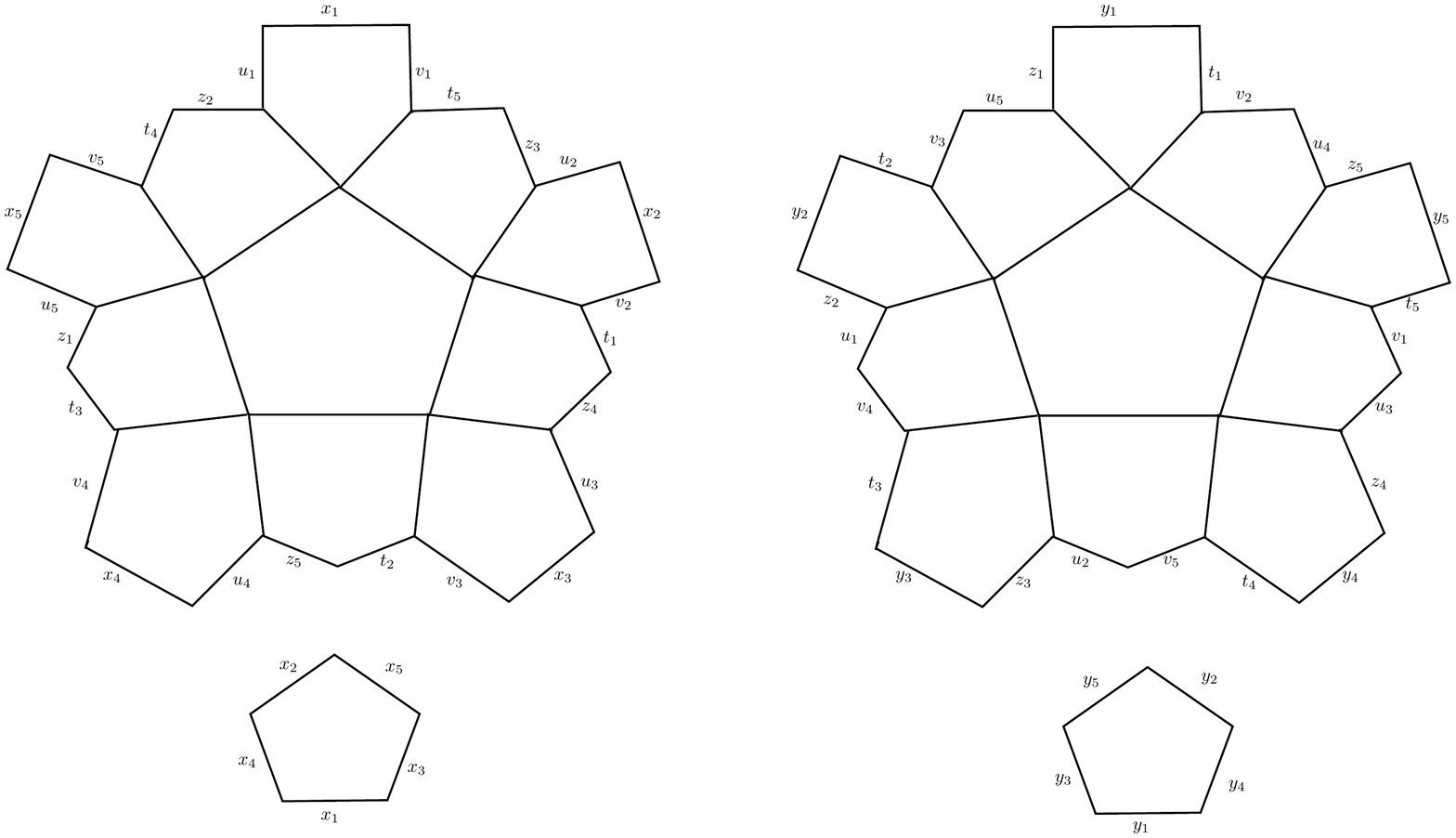} 
			
			\caption{The dessin $\mathcal J$} \label{F11}
		\end{figure}
		
		The new figure obtained from the scanning at Figure~\ref{F8a} by cutting the edges marked by $x_i,y_i,u_i, v_i,z_i,t_i$, $i=1,2,3,4,5$, and gluing the edges marked by $p_j,q_j,r_j,s_j$, $j=1,2,3,4,5$, is drawn at Figure~\ref{F11}. So, this is another gluing the dessin $\mathcal J$. We do not draw the edges of $I_4$ any more.
	\end{proof}
	
	\begin{lemma} \label{l3}
		The dessin $\mathcal J$ is isomorphic to $\mathcal D$.
	\end{lemma}
	
	\begin{proof}
		
		Figure \ref{F11} contains the scanning of the dessin $\mathcal J$. The same letters denote the sides to be glued. Gluing   two sides marked by $x_1$ to each other and two sides marked by $y_1$ to each other we obtain Figure~\ref{F9}.
		
		\begin{figure}[htbp]
			\includegraphics[trim=0pt 70pt 0pt 122pt,clip,width=1\textwidth]{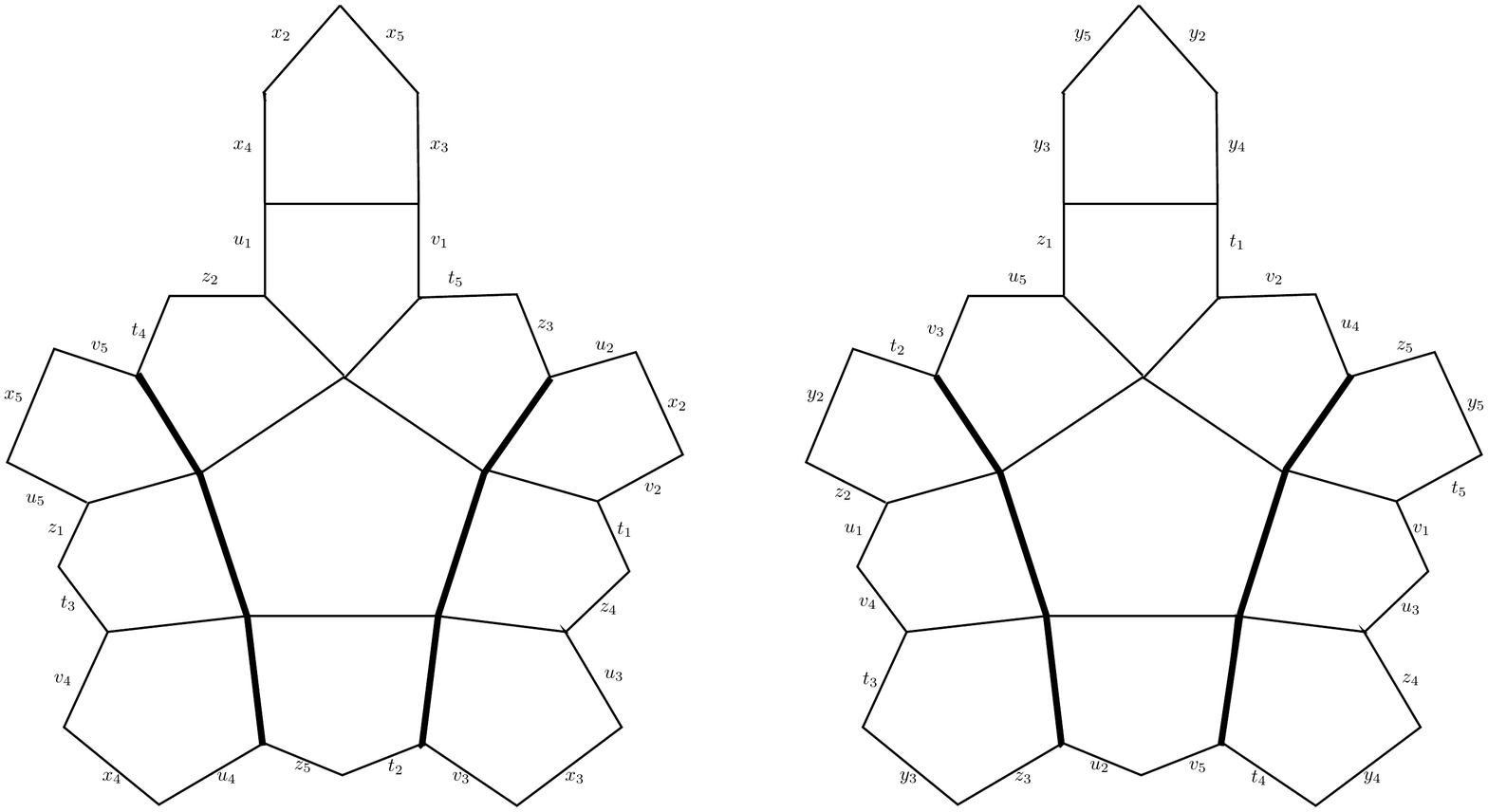} 
			\caption{Another scanning of the dessin   $\mathcal J$ with the lines to cut off.} 	\label{F9}
		\end{figure}
		
		Let us cut  the polygons at Figure \ref{F9} by the dark lines. Then we glue the sides marked by the letters $x_5$ to each other,  the sides marked by the letters $t_3$ to each other, and the sides marked by the letters $v_4$ to each other. Thus we obtain the upper polygon from Figure \ref{fig5_5}. Now let us glue the sides marked by the letters  $y_5$,  the sides marked by the letters $u_3$, and the sides marked by the letters $z_4$. Thus we obtain the lower polygon from Figure~\ref{fig5_5}.  
	\end{proof}
	
	\begin{lemma} \label{Main}
		The dessin ${\mathcal D}$ corresponding to the orientation covering of $\overline{{\mathcal M}_{0,5}^{\mathbb R}}$ is the preimage of $[1,\infty)$ under the map $f_{B_5}:B_5\to \bP$. Here preimages of $\infty$ are middle points of the edges.
	\end{lemma}
	\begin{proof}
		By Lemma   \ref{l3} the dessin $\mathcal J$ is isomorphic to the dessin $\mathcal D$. By its definition, ${\mathcal J}$ 
		is obtained from the dessin $(I_4\cup I_4^*)^*$ by   the color exchange and forgetting the vertices of valency 2.  Lemma~\ref{L0} concludes the proof.
	\end{proof}

	\begin{theorem}
		The Belyi pair of the dessin d'enfant ${\mathcal D}$ provided by the cell decomposition of ${\mathcal L}(\overline{{\mathcal M}_{0,5}^{\mathbb R}})$  is $$\Bigl(B_5,1-\frac1{f_{B_5}}\Bigr)= \Bigl(B_5,  \frac{3125 (\frac1{x_1}+\frac1{x_2}+ \frac1{x_3}+\frac1{x_4} +\frac1{x_5})^4}{256\, x_1x_2x_3x_4x_5} \Bigr).$$
	\end{theorem}

	
	\begin{proof} The direct consequence of Lemma~\ref{L0} applied to the result of Lemma~\ref{Main}. 
	\end{proof}
	

	\begin{corollary}
		There exists an   action of symmetric group $S_5$ on
		${\mathcal L}(\overline{{\mathcal M}_{0,5}^{\mathbb R}})$ such that
		all elements of $S_5$ act as orientation preserving diffeomorphisms respectful to cell decomposition. 
	\end{corollary}
	\begin{proof} This statement is a partial case of  Lemma~\ref{Sn}, however,
		we provide here an independent direct proof which follows from the Belyi theory and is valid only for~$n=5$. 
		
		The permutation of the coordinates $(x_1:\ldots:x_5)$ is a holomorphic automorphism of the Bring curve $B_5$. Therefore, the group $S_5$ acts on $B_5$ by orientation preserving diffeomorphisms. It is straightforward to see that these diffeomorphisms preserve the Belyi function. Therefore elements of the group $S_5$ act on the Belyi pair, and hence, on  the dessin $\mathcal D$, which is ${\mathcal L}(\overline{{\mathcal M}_{0,5}^{\mathbb R}})$. 
	\end{proof}	
	
	\begin{corollary}
		The symmetric group $S_5$ acts   transitively  on the  dessin~${\mathcal D}$ and the  dessin   ${\mathcal D}$ is regular.  \end{corollary}
	
	\begin{proof}
		Follows by Lemma   \ref{l3} from the fact that ${\mathcal J}$ is regular.
	\end{proof}
	
	\section*{Acknowledgments}
	
	
	The authors would like to express their deep gratitude to \fbox{Sergey Natanzon},  George Shabat and Sasha Zvonkin for fruitful comments and motivating discussions, and to Grisha Guterman for the help with drawing pictures.

	The work of Natalia Amburg was funded by the Russian Science Foundation  
	(Grant No. 21-12-00400). The work of Elena Kreines was partially funded by the ISF Grant  No. 1092/22.

	N. Ya. Amburg:
	\\
	NRC "Kurchatov institute", Moscow,  123182,
	Russia
	\\
	Faculty of Mathematics, National Research University Higher School of  
	Economics,  Moscow,  119048, Russia
	\\
	Institute for Information Transmission Problems, Moscow, 127051, Russia
	\\
	e-mail: amburg@mccme.ru
	
	E. M. Kreines:
	\\
	Tel Aviv University, Tel Aviv,  6997801,
	Israel
	\\        
	Ben Gurion University of Negev, Beer-Sheva, 8410501, Israel
	\\
	e-mail: 	kreines@tauex.tau.ac.il
\end{document}